\documentclass{jgcc} 
\pdfoutput=1
\usepackage{lastpage}
\jgccdoi{16}{1}{4}{13872}  
\jgccheading{}{\pageref{LastPage}}{}{}{Apr.~26,~2024}{Jul.~12,~2024}{}    


\usepackage{hyperref}
\usepackage{xfrac}
\usepackage{amssymb}
\theoremstyle{plain} 

\def\x{X}
\def\a{a}
\def\L{\mathcal{L}}
\def\A{\mathcal{A}}
\def\La{\L_\A}
\def\N{\mathbb{N}}
\def\P{\mathcal{P}}

\usetikzlibrary{calc,positioning}

\tikzstyle{every node}=[circle, draw, fill=black!50, inner sep=0pt, minimum width=4pt]
\tikzstyle{red node}=[circle, draw, fill=red!80, inner sep=0pt, minimum width=4pt]
\tikzstyle{green node}=[circle, draw, fill=green!50, inner sep=0pt, minimum width=4pt]
\tikzstyle{blue node}=[circle, draw, fill=blue!50, inner sep=0pt, minimum width=4pt]
\tikzstyle{yellow node}=[circle, draw, fill=yellow!50, inner sep=0pt, minimum width=4pt]
\tikzstyle{blue box}=[rectangle, minimum height=0.3cm,draw, fill=blue!50, inner sep=0pt, minimum width=0.3cm]
\tikzstyle{red box}=[rectangle, minimum height=0.3cm, draw, fill=red!50, inner sep=0pt, minimum width=0.3cm]

\begin{document}

\title{On equationally Noetherian predicate structures}

\author[I.~M.~Buchinskiy]{Ivan Buchinskiy}	
\address{Sobolev Institute of Mathematics of SB RAS, Omsk, Russia}	
\email{buchvan@mail.ru}  

\author[M. V. Kotov]{Matvei Kotov}	
\address{Sobolev Institute of Mathematics of SB RAS, Omsk, Russia}	
\email{matvej.kotov@gmail.com}  

\author[A. V. Treier]{Alexander Treier}	
\address{Sobolev Institute of Mathematics of SB RAS, Omsk, Russia}	
\email{alexander.treyer@gmail.com}  





\begin{abstract}
  \noindent In this paper, we prove a criterion for a predicate structure to be equationally Noetherian.
\end{abstract}

\maketitle

\section{Introduction}
    Algebraic geometry over algebraic structures is a branch of mathematics that lies at the intersection of algebra and model theory. The main objects of study in this theory are algebraic sets, i.e., sets of solutions to systems of equations. Researchers also try to find patterns that are common to classes of algebraic structures and to generalize results that are true for specific algebraic structures to arbitrary algebraic structures. This theory arose in papers by Plotkin~\cite{Plotkin1997, Plotkin2003} for varieties of algebras and in a series of papers by Daniyarova, Miasnikov and Remeslennikov started in~\cite{DaniyarovaMyasnikovRemeslennikov2008, DaniyarovaMyasnikovRemeslennikov2012_1, DaniyarovaMyasnikovRemeslennikov2011} and subsequently published as a book~\cite{DaniyarovaMyasnikovRemeslennikov2016}. This theory was preceded by work on algebraic geometry over groups~\cite{BaumslagMyasnikovRemeslennikov1999, MyasnikovRemeslennikov2000}.

    In classical algebraic geometry over associative rings and in universal algebraic geometry over arbitrary algebraic structures, one of the very important roles plays the property of being equationally Noetherian. Recall that an algebraic structure is called \textit{equationally Noetherian} if for every finite set of variables $\x$, every system of equations $S(\x)$ is equivalent to a finite subsystem $S_0(\x) \subseteq S(\x)$. The importance of this property is related to the so-called unification theorems~\cite{DaniyarovaMyasnikovRemeslennikov2008}, which, for equationally Noetherian algebras, reduce problems of classifying algebraic sets to logical problems of describing some quasivarieties and universal classes. Also, in algebraic structures with this property, we can study only finite systems of equations. Additionally, algebraic sets can be decomposed into finite unions of irreducible sets. Another advantage is that such structures are good for calculating the dimension of algebraic sets~\cite[Chapter 6]{DaniyarovaMyasnikovRemeslennikov2016}.

    There are many examples of algebraic structures with and without this property. For example, all finite algebraic structures, abelian groups, linear groups over a Noetherian ring, and torsion-free hyperbolic groups are equationally Noetherian. On the other hand, some infinitely generated nilpotent groups,  wreath products of a non-abelian group and an infinite cyclic group, infinite direct products of non-abelian groups, and minimax algebraic structures are not equationally Noetherian \cite{BaumslagMyasnikovRomankov1997, BaumslagMyasnikovRemeslennikov1999, ShahryariShevlyakov2017, GuptaRomanovskii2007, DvorzhetskiyKotov2008}. 

    There are also several works devoted to the property of being equationally Noetherian and its generalizations~\cite{ModabberiShahryari2016, Kotov2011, Kotov2013_2, Shevlyakov2011}.

    In recent years, researchers in universal algebraic geometry have focused on algebraic structures with predicates. The theoretical foundations of this direction have been developed in the paper~\cite{DaniyarovaMyasnikovRemeslennikov2012_2}. Iljev and Remeslennikov~\cite{IljevRemeslennikov2017} and Buchinskiy and Treier~\cite{BuchinskiyTreier2023} have studied systems of equations over graphs. Shevlyakov has studied algebraic geometry over groups in a predicate language~\cite{Shevlyakov2018}, equations over direct powers of algebraic structures in relational languages~\cite{Shevlyakov2021}, and algebraic geometry over algebraic structures with the relation $\neq$~\cite{Shevlyakov2016}. Dvorzhetskiy~\cite{Dvorzhetskiy2013} has considered lattices with a finite collection of predicate symbols. Partially ordered sets have been considered in papers by Nikitin and Shevlyakov~\cite{NikitinShevlyakov2020, NikitinShevlyakov2021} and by Nikitin and Kudyk~\cite{NikitinKudyk2018}.

    An algebraic structure in a language without functional symbols is called a \textit{predicate structure}. Let a language $\L$ consist of a finite number of predicates and a finite number of constants, and let $\x$ be a finite set of variables. Then, there is only a finite number of nonequivalent systems of equations in the variables $\x$ in the language $\L$. It is easy to see that all such $\L$-structures are equationally Noetherian. Therefore, it is natural to ask whether such an $\L$-structure is equationally Noetherian or not in a language $\L'$ extended by an infinite number of constants. Let $\L$ be a language, and $\A = \left<A, \L\right>$ be an $\L$-structure. Denote by $\L_\A$ the language obtained from $\L$ by adding a new constant symbol $c_a$ for every element $a \in A$. All these constant symbols will be interpreted by the corresponding constants. Algebraic geometry in the language $\L_\A$ is called Diophantine. This case will be considered in this work. Also, by default, we will assume that every language includes the equality predicate $=$.

    Previously, in \cite{BuchinskiyTreier2023}, all equationally Noetherian graphs were described in terms of forbidden subgraphs. In this paper, developing and generalizing ideas from~\cite{BuchinskiyTreier2023}, we give a description of predicate algebraic structures with a finite number of predicates that are not equationally Noetherian (see Theorem~\ref{theorem_1P}).

    In the paper~\cite{Kotov2013_1}, a criterion for an arbitrary algebraic structure without predicates not to be equationally Noetherian was given. Note that that work was inspired by~\cite{BaumslagMyasnikovRomankov1997}. It turns out that the criterion is also true for algebraic structures with predicates. The proof of the generalized criterion is very similar to the proof of the original criterion and can be found in~\cite{BuchinskiyTreier2023}.   

    \begin{lem}\label{lemma:kotov}
        An algebraic structure $\A = \langle A, \L \rangle$ is not equationally Noetherian if and only if there is a sequence of elements $(\a_i)_{i \in \N}$, $\a_i \in A^n$, and a sequence of $\L$-equations $(s_i(\x))_{i \in \N}$, $\x = \{x_1, \ldots, x_n\}$, such that
        \begin{equation}\label{cond_from_lemma}
            \A \not \models s_i(\a_i) \text{ for all } i \in \N \text{, and } \A \models s_j(\a_i) \text{ for all } j < i.
        \end{equation}
    \end{lem}	

    \begin{figure}
        \centering
	    \begin{tikzpicture}[thick,scale=1, every node/.style={scale=1}]
            \path (-1,0) node[red node][label=180:$\a_1$] (a1) {};
            \path (1,0) node[blue node][label=0:$s_1(\x)$] (b1) {};
            \path (-1,-1) node[red node][label=180:$\a_2$] (a2) {};
            \path (1,-1) node[blue node][label=0:$s_2(\x)$] (b2) {};
            \path (-1,-2) node[red node][label=180:$\a_3$] (a3) {};
            \path (1,-2) node[blue node][label=0:$s_3(\x)$] (b3) {};
            \path (-0.5,-3) node[][label=180:\rotatebox{125}{$\ddots$}] (ldots1) {};
            \path (1.5,-3) node[][label=180:\rotatebox{125}{$\ddots$}] (rdots1) {};
            \path (-1,-4) node[red node][label=180:$\a_n$] (an) {};
            \path (1,-4) node[blue node][label=0:$s_n(\x)$] (bn) {};
            \path (-0.5,-5) node[][label=180:\rotatebox{125}{$\ddots$}] (ldots2) {};
            \path (1.5,-5) node[][label=180:\rotatebox{125}{$\ddots$}] (rdots2) {};
            \draw (a1) -- (b1) [dashed];
            \draw (a2) -- (b1);
            \draw (a2) -- (b2) [dashed];
            \draw (a3) -- (b1);
            \draw (a3) -- (b2);
            \draw (a3) -- (b3) [dashed];
            \draw (an) -- (b1);
            \draw (an) -- (b2);
            \draw (an) -- (b3);
            \draw (an) -- (bn) [dashed];
        \end{tikzpicture}
        \caption{\it An illustration to Lemma~\ref{lemma:kotov}}
    \end{figure}
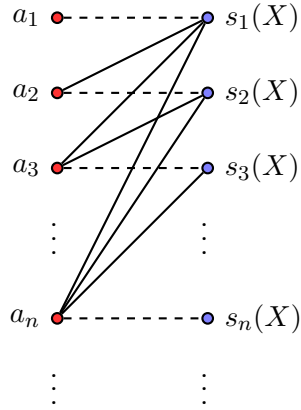

    Lemma \ref{lemma:kotov} gives us a universal description of algebraic structures that are not equationally Noetherian in terms of the satisfiability of atomic formulas in the language $\L$. The goal of this work is to translate the condition of not being equationally Noetherian for predicate structures from the language of the satisfiability of atomic formulas to the language of forbidden substructures.

\section{Preliminaries}

    Let us recall some basic concepts of algebraic geometry over algebraic structures that we will need later. We will follow the book \cite{DaniyarovaMyasnikovRemeslennikov2016}.

    A language $\L = \{P^{(n_1)}_1, \ldots, P^{(n_k)}_k\}$, where each $P^{(n_i)}_i$ is an $n_i$-ary predicate symbol, is called a \textit{predicate} language. If $\L$ and $\L'$ are two languages, and $\L \subseteq \L'$, then $\L$ is called a \textit{reduction} of $\L'$, and $\L'$ is called an \textit{expansion} of $\L$.

    Let $\A = \langle A, \L \rangle$ be an arbitrary algebraic structure in the language $\L$. The extended language $\La = \L \cup \{c_a \mid a \in A\}$ obtained from $\L$ by adding a new constant for every $a \in A$ is called a \textit{language with constants from} $A$. In this paper, by default, we will consider only such languages. Therefore, unless otherwise stated, $\A = \langle A, P^{(n)} \rangle$ will stand for an algebraic structure $\A$ with the underlying set $A$ in the language with the $n$-ary predicate symbol $P$ and the constants from $A$.

    For a language $\La$, every equation has one of the following forms: 
    \begin{enumerate}
        \item
            $P_i(w_1, \ldots, w_{n_i})$, where $i \in \{1, \ldots, k\}$, and, for all $j = 1, \ldots, n_i$, the term $w_j$ is either a constant of the language $\La$ or a variable;
        \item
            $w_1 = w_2$, where each term $w_i$, $i = 1, 2$, is either a constant of the language $\La$ or a variable.
    \end{enumerate}

    Consider a structure $\A = \langle A, \L \rangle$. A point $\a \in \A^n$ is called a \textit{solution} to an equation $s(\x)$ in $\La$ in $n$ variables $\x = \{x_1, \ldots, x_n\}$ over the structure $\A$ if $\A \models s(\a)$. A point $\a \in \A^n$ is called a \textit{solution} to a system of equations $S(\x)$ over the algebraic structure $\A$ if $\a$ is a solution to every equation of the system $S(\x)$. The set of all solutions to the system of equations $S(\x)$ is called an \textit{algebraic set} over $\A$ and is denoted by $V_{\A}(S(\x))$. Two systems of equations $S_1(\x)$ and $S_2(\x)$ in language $\L$ are called \textit{equivalent} over $\A$ if their solutions coincide.

    An algebraic structure $\A$ is called \textit{equationally Noetherian} if, for every positive integer $n$, every system of equations $S(\x)$ in $n$ variables $\x$ is equivalent to a finite subsystem $S_0(\x) \subseteq S(\x)$. 

    Equations with no variables are either always true or always false. Systems of such equations can be replaced by one false or true equation. Therefore, we will not consider systems of equations containing an infinite number of equations with no variables.

    In the next sections, we will need the following corollary from Lemma~\ref{lemma:kotov}.

    \begin{cor}[\cite{BuchinskiyTreier2023}]\label{cor:kotov}
        Let $\A$ be an algebraic structure, $\x = \{x_1, \ldots, \allowbreak x_n\}$ be a finite set of variables, and $S(\x)$ be a system of equations that is not equivalent over~$\A$ to any of its finite subsystems. Then:
        \begin{enumerate}
            \item
                There is an infinite subsystem $S' = \{s_1(\x), \ldots, s_i(\x), \ldots\}$ and a sequence of elements $(\a_i)_{i \in \N}$, $\a_i \in \A^n$ such that~(\ref{cond_from_lemma}) holds. This system is also not equivalent over $\A$ to any of its finite subsystems.
            \item
                Also, for every infinite subsystem $S'' \subset S'$ and the corresponding subsequence $(\a'_j)_{j \in \N}$ of $(\a_i)_{i \in \N}$, the condition (\ref{cond_from_lemma}) also holds. Therefore, $S''$ is also not equivalent to any of its finite subsystems.
        \end{enumerate}
    \end{cor}

\section{Equationally Noetherian predicate algebraic structures}

    In this section, we formulate and prove a criterion for predicate algebraic structures to be equationally Noetherian.

    \subsection{Configurations of predicate equations}

        The following proposition holds for predicate algebraic structures with a finite number of predicates.

        \begin{prop}\label{state:onevariable}
            Let $\L$ be a predicate language with a finite number of predicate symbols. Denote these predicates by $P_1^{(n_1)}, \ldots, P_m^{(n_m)}$, and denote the set of constant symbols by $C$. Let $X$ be a finite set of variables, and $S(X)$ be an infinite system of equations in $\L$. Let $S(X)$ not contain infinite subsystems of equations with no variables. Then at least one of the following conditions is satisfied:
            \begin{itemize}
                \item
                    there is a predicate symbol $P_i^{n_i} \in \L$ and there is an infinite subsystem of equations $S_{P_i} \subseteq S$ consisting of equations of the form $P_i(w_1, \ldots, w_{n_i})$, where $w_j \in X \cup C$;
                \item 
                    there is an infinite subsystem of equations $S_{=} \subseteq S$ consisting of equations of the form $w_1 = w_2$, where either $w_1 \in X$ and  $w_2 \in C$ or $w_1 \in C$ and  $w_2 \in C$.
            \end{itemize}
        \end{prop}

        \begin{proof}
            Let $S = S_C \cup S_X \cup S_{\x, C}$, where the subsystem $S_{\x, C}$ consists of equations that have variables and constants at the same time, $S_X$ consists of equations with no constants, and $S_C$ consists of equations with no variables. Because the number of variables is finite, the subsystem $S^X$ is finite. The fact that $S_C$ is finite follows from the assumption of the proposition. Therefore, the subsystem $S_{\x, C}$ is infinite. Let $S_{\x,C} = S_{P_1} \cup S_{P_2} \cup \ldots \cup S_{P_m} \cup S_{=}$,  where $S_{P_i}$ consists of equations  of the form $P_i(w_1, \ldots, w_{n_i})$, and $S_{=}$ consists of $w_1 = w_2$. Since $S_{\x, C}$ is infinite, at least one of the subsystems $S_{P_1}, S_{P_2}, \ldots, S_{P_m}, S_{=}$ is infinite.
        \end{proof}

        \begin{rem}
            Let a predicate language $\L$ consist of constants and one unary predicate. Then, every $\L$-structure is equationally Noetherian.
        \end{rem}

        Let $\L$ be a predicate language. Let $Q_1^{(n)}$ and $Q_2^{(k)}$ be predicate symbols of $\L$, $C$ be the set of constant symbols of $\L$, and $X$ be a finite set of variables. We will say that two equations $Q_1(v_1, \ldots, v_n)$ and $Q_2(w_1, \ldots, w_k)$, where  $v_i, w_j \in X \cup C$, \textit{have the same configuration} if the following conditions hold: 
        \begin{enumerate}
            \item 
                the predicate symbols $Q_1^{(n)}$ and $Q_2^{(k)}$ coincide, and, therefore, $n = k$; 
            \item 
                for all $i = 1, \ldots, n$, either $v_i$ and $w_i$ are constants, or $v_i$ and $w_i$ are the same variable.
        \end{enumerate}

        \begin{rem}\label{comment:finite}
            For a predicate symbol $Q$ and a finite set of variables $X$, there are only a finite number of pairwise different configurations of equations of the form $Q(v_1, \ldots, v_n)$ in the variables $X$.
        \end{rem}

        The following lemma allows us to consider only predicate equations that have the same configuration.

        \begin{lem}\label{theorem:basic}
            Let $\A = \langle A, \La \rangle$ be a predicate structure with constants from $A$ and a finite number of predicates, and let $\A$ not be equationally Noetherian. Then, there are infinite sequences of elements $\{(a_1^i, \ldots, a_p^i)\}_{i \in \N}$ and equations $S = \{s_i(X)\}_{i \in \N}$ of the same configuration in a finite set of variables $X$ such that~(\ref{cond_from_lemma}) holds.
        \end{lem}

        \begin{proof}
            Since $\A$ is not equationally Noetherian, it follows from Corollary~\ref{cor:kotov} that there are sequences of elements $(a_i)_{i \in \N}$ and of equations $S(\x)$ such that~(\ref{lemma:kotov}) holds. From Proposition~\ref{state:onevariable}, there is an infinite subsystem $S' \subseteq S$ such that all its equations contain one predicate $P^{(n)}$ of $\La$, i.e., $S' = \{P(w_1^i, \ldots, w_{n}^i)\}_{i \in \N}$, where $w_j^i \in X \cup A$. According to Remark~\ref{comment:finite}, there is an infinite subsystem $S'' \subseteq S'$ such that it is not equivalent to any of its finite subsystems and consists of predicate equations of the same configuration.
        \end{proof}

        It follows from Lemma~\ref{theorem:basic} that, without loss of generality, we can assume that all equations of the form $P(v_1, \ldots, v_n)$, where $P^{(n)}$ is a predicate symbol of $\La$, have the form $P(x_1, \ldots,x_p, b_1, \ldots, b_t)$, where $p + t = n$, $x_1, \ldots, x_p$ are variables, and $b_1, \ldots, b_t$ are constant symbols of the language $\La$.

    \subsection{Projections of predicates and gluings of predicates}
    
        In this subsection, we give two ways to construct a new predicate from an existing predicate.

        \begin{defi}
            Let $P^{(n)}$ be an arbitrary $n$-ary predicate, $I = \{i_1, i_2, \ldots,\allowbreak i_k\} \subset \{1, 2, \ldots n\}$, $0 < k < n$, and $J = \{1, 2, \ldots n\} \setminus I = \{l_1, l_2, \ldots, l_{n - k}\}$. A predicate $P'^{(k)}$ is called the \textit{projection of the predicate $P$ onto the set of components $I$ by using elements $p_{1}, p_{2}, \ldots, p_{n-k} \in A$} if, for all $a_{1}, a_{2}, \ldots, a_{k} \in A$, 
            \begin{equation*}
                \A \models P'(a_{1}, \ldots, a_{k}) \iff \A \models P(c_1, \ldots, c_n),
            \end{equation*}
            where $c_{i_j} = a_j$ if $i_j \in I$, for all $j = 1,2, \ldots, k$, and $c_{l_j} = p_j$ if $l_j \in J$ for all $j = 1, 2, \ldots, n - k$.
        \end{defi}

        In other words, a projection of a predicate $P$ is fixing some arguments of the predicate $P$ by given constants from the underlying set of the algebraic structure.

        We use the name ``a projection of a predicate'' because it is ideologically similar to the notion of a projection of a relation from database theory~\cite{Date2003}.

        \begin{exa}\label{example_projection}
            Let $\Gamma = \langle \{v_1, v_2, v_3, v_4, v_5\}, E^{(3)} \rangle$ be a hypergraph with 5 nodes, where $3$-hyperedges are triples $(v_1, v_1, v_2), (v_1, v_3, v_3), (v_2, v_4, v_1),\allowbreak (v_3, v_2, v_2), (v_5,\allowbreak v_4, v_5)$. Then, for example, the projection of the predicate $E$ onto the set of components $I = \{1, 3\}$ by the element $v_4$ is the binary predicate that is true only for pairs $(v_2, v_1), (v_5, v_5)$.
        \end{exa}

        \begin{defi}
            Let $P^{(n)}$ be an $n$-ary predicate, and $I = \bigsqcup_{j = 1}^m I_j$ be an exact partition of the set $\{1, \ldots, n\}$. A predicate $\sfrac{P}{I}^{(m)}$, $m \leq n$, is called the \textit{gluing of the predicate $P$ by the partition $I$} if the following condition holds:
            \begin{equation*}
                \A \models \sfrac{P}{I}(a_1, \ldots, a_m) \iff \A \models P(b_1, \ldots, b_n),
            \end{equation*}   
            where $b_i = a_k$ if and only if $i \in I_k$.
        \end{defi}

        \begin{exa}
            Consider the hypergraph $\Gamma$ from Example~\ref{example_projection}. Let $I = \{\{1\}, \{2, 3\}\}$ be an exact partition of the set $\{1, 2, 3\}$. Then $\sfrac{E}{I}$ is the binary predicate that is true only for pairs $(v_1, v_3), (v_3, v_2)$.
        \end{exa}

    \subsection{Perfectly non-Noetherian structures}
    
        The notion of a perfectly non-Noetherian substructure plays a key role in our criterion for structures to be equationally Noetherian. In fact, this object is a forbidden substructure for the equationally Noetherian property.

        \begin{defi}\label{def_cns} 
		    Let $P$ be an $n$-ary predicate symbol of a language $\La$. We  say that an algebraic structure $\A = \langle A, \La \rangle$ contains a {\it $P$-perfectly non-Noetherian substructure} if there are sequences of elements $\{(a^i_1, \ldots, a^i_p)\}_{i \in \N}$ and equations $\{P(x_1, \ldots, x_p, b^i_1, \ldots, b^i_t)\}_{i \in \N}$ such that 
            \begin{enumerate}
                \item 
                    $p + t = n$;
                \item
                    $a^1_1, \ldots, a^1_p,  b^1_1, \ldots b^1_t, \ldots, \allowbreak a^k_1, \ldots, a^k_p,  b^k_1, \ldots, b^k_t, \ldots$  are pairwise different;
                \item 
                    $\A \not\models P(a^i_1, \ldots, a^i_p, b^i_1, \ldots, b^i_t)$ for all $i$;
                \item
                    $\A \models P(a^i_1, \ldots, a^i_p, b^j_1, \ldots, b^j_t)$ for all  $j < i$.
            \end{enumerate}
            To be short, we will sometimes say ``a \textit{completely non-Noetherian substructure}'' instead of ``a $P$-completely non-Noetherian substructure'' if $\La$ has only one predicate symbol $P$.
        \end{defi}

        \begin{exa}
            A base non-Noetherian graph mentioned in the paper~\cite{BuchinskiyTreier2023} contains a perfectly non-Noetherian substructure. Also, note that the property of containing a perfectly non-Noetherian substructure for graphs coincides with the notion to be a perfectly non-Noetherian graph from \cite{BuchinskiyTreier2023}.
    
            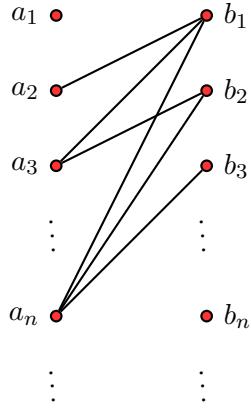
\begin{figure}
	            \centering
                \begin{tikzpicture}[thick,scale=1, every node/.style={scale=1}]
                    \path (-1,0) node[red node][label=180:$a_1$] (a1) {};
                    \path (1,0) node[red node][label=0:$b_1$] (b1) {};
                    \path (-1,-1) node[red node][label=180:$a_2$] (a2) {};
                    \path (1,-1) node[red node][label=0:$b_2$] (b2) {};
                    \path (-1,-2) node[red node][label=180:$a_3$] (a3) {};
                    \path (1,-2) node[red node][label=0:$b_3$] (b3) {};
                    \path (-0.5,-3) node[][label=180:\rotatebox{130}{$\ddots$}] (ldots1) {};
                    \path (1.5,-3) node[][label=180:\rotatebox{130}{$\ddots$}] (rdots1) {};
                    \path (-1,-4) node[red node][label=180:$a_n$] (an) {};
                    \path (1,-4) node[red node][label=0:$b_n$] (bn) {};
                    \path (-0.5,-5) node[][label=180:\rotatebox{130}{$\ddots$}] (ldots2) {};
                    \path (1.5,-5) node[][label=180:\rotatebox{130}{$\ddots$}] (rdots2) {};
                    \draw (a2) -- (b1);
                    \draw (a3) -- (b1);
                    \draw (a3) -- (b2);
                    \draw (an) -- (b1);
                    \draw (an) -- (b2);
                    \draw (an) -- (b3);
                \end{tikzpicture}
                \caption{\it The base non-Noetherian graph}
            \end{figure}
        \end{exa}

        Let us highlight the following special sort of non-Noetherian structures that arise for irreflexive binary predicates. 

        \begin{defi}\label{def:clique}
            We say that a predicate algebraic structure $\A = \langle A, P^{(2)} \rangle$ \textit{contains a non-Noetherian clique} if there is a sequence of elements $\{a_i\}_{i \in \N} \subseteq A$ such that $\A \not\models P(a_i, a_i)$ for all $i$ and $\A \models P(a_i, a_j)$ for all $j < i$. 
        \end{defi}

        This notion is similar to the notion of a perfectly non-Noetherian substructure introduced above. The difference is that the sequences of elements $\{a_i\}_{i \in \N}$ and $\{b_i\}_{i \in \N}$ from Definition~\ref{def_cns} coincide.
 
        \begin{rem}\label{comment:klika}
            If an algebraic structure $\A = \langle A, P^{(2)} \rangle$ has a non-Noetherian clique, then it is not equationally Noetherian in the language with constants. Indeed, the sequences of elements $\{a_i\}_{i\in\N}$ and equations $\{P(x, a_i)\}_{i \in \N}$ satisfy the condition~(\ref{cond_from_lemma}) from Lemma~\ref{lemma:kotov}.
        \end{rem}

        An example of a non-Noetherian clique is a countable clique for simple graphs. It is easy to see that, for simple graphs, every infinite clique is a non-Noetherian graph.

    \subsection{Criterion for a predicate structure to be equationally Noetherian}

        Let $D = \{(d_1^i, \ldots, d_n^i)\}_{i \in \N}$ be an arbitrary sequence of tuples of length $n$ of elements of some infinite set $A$. Denote by $C_j(D)$ the set of elements in the $j$-th column.

        In this subsection, we will need the following technical lemma.

        \begin{lem}\label{lemma_tuple}
            Let $D = \{(d_1^i, \ldots, d_n^i)\}_{i \in \N}$ be a sequence of tuples of elements of some infinite set $A$. Then, there is a subsequence $\Tilde{D} \subseteq D$ such that
            \begin{enumerate}
                \item  
                    for each $j = 1, 2, \ldots, n$, $C_j(\Tilde{D})$ contains either only one element or infinitely many pairwise different elements;
                \item 
                    any two columns either coincide or have no common elements.
            \end{enumerate}
        \end{lem}
        
        \begin{proof}
            If $|C_1(D)| < \infty$, then choose a subsequence $D_1$ of $D$ such that $|C_1(D_1)| = 1$. If $|C_1(D)| = \infty$, then choose a subsequence $D_1$ of $D$ such that all the elements in the column $C_1(D_1)$ are pairwise different.  Perform this procedure for the second column of the sequence $D_1$, i.e., choose from $D_1$ a subsequence $D_2$ such that $C_2(D_2)$ consists of either only one element or pairwise different elements. Perform this procedure for all the columns. After that, we obtain a subsequence $D_n = \{(d_1^{i}, d_2^{i}, \ldots, d_n^{i})\}_{i \in \N}$ of $D$ such that every column either consists of one element or all the elements are pairwise different. Without loss of generality, we will assume that there is no column consisting of one element in the sequence $D_n$. Now, let us prove that it is possible to choose a subsequence of  $D_n$ such that every two columns either coincide or have no common elements.

            Let $n = 2$. If the subsequence $D_2$ has infinitely many rows in which the elements of two columns $d_1^{i}$ and $d_2^{i}$ are equal, then choose the $D_{2, =}$ consisting of all these rows. Let there be only a finite number of such rows. Then, consider the subsequence $D_{2, \ne}$ that does not contain the rows such that $d_1^{i} = d_2^{i}$. Consider the first row $(d_1^{1}, d_2^{1})$ of $D_{2, \ne}$. If the element $d_1^{1}$ is in the sequence $D_{2, \ne}$, then it can be only in the second column and at most one time. Let it be in the $j$-th row. Then, remove this row from $D_{2, \ne}$. Repeat the same procedure for the element $d_2^{1}$, i.e., if it is in the $k$-th row in the first column, then, if $k \ne j$, remove this row from $D_{2, \ne}$. Go to the next row in $D_{2, \ne}$. Performing a similar procedure, we remove at most two rows from $D_{2, \ne}$. Note that the first row will not be removed. Performing this procedure for all the rows, we obtain the subsequence $D_{2, \ne}$ with only pairwise different elements.

            If $n > 2$, then we perform the procedure described above for all the pairs of columns. At each step for a selected pair of columns in the current subsequence $D'$ of $D_n$, we choose a new subsequence $D''$ and, for the next pair of columns, we perform the procedure on the subsequence $D''$. After performing the procedure on all the pairs of columns, we obtain a subsequence $\Tilde{D}$ that satisfies the condition of the lemma.
        \end{proof}

        To prove the main result of this paper, Theorem~\ref{th:main}, we need the following lemma. This lemma is a criterion for a predicate structure with one predicate to be equationally Noetherian.

        \begin{lem}\label{theorem_1P}
            An algebraic structure $\A = \langle A, P^{(n)} \rangle$ in a language with one predicate symbol $P^{(n)}$ and constants from $A$ is not equationally Noetherian if and only if there is a projection $P'^{(k)}$ of the predicate $P^{(n)}$ and an exact partition $I$ of $\{1, \ldots, k\}$ such that at least one of the following conditions is true:
            \begin{itemize}
                \item
                    $|I| > 1$ and the algebraic structure $\A' = \langle A, \sfrac{P'}{I} \rangle$ has a perfectly non-Noetherian substructure;
                \item 
                    $|I| = 1$ and the algebraic structure $\A' = \langle A, Q \rangle$, where $Q = \sfrac{P'}{\{\{1\}, \{2, \ldots, k\}\}}$, has a non-Noetherian clique.
            \end{itemize}
        \end{lem}

        \begin{proof}
            Since $\A$ is not equationally Noetherian,  it follows from Lemma~\ref{theorem:basic} that there is a sequence of elements $\{(a_1^i, \ldots, a_p^i)\}_{i \in \N} \subseteq A^p$ and a sequence of equations $\{P(x_1, \ldots, x_p, b_1^i, \ldots, \allowbreak b_t^i)\}_{i \in \N}$ such that $\A \not\models P(a_1^i,\allowbreak \ldots, a_p^i, b_1^i, \ldots, b_t^i)$ for all $i$ and $\A \models P(a_1^i, \ldots, a_p^i, b_1^j, \ldots, b_t^j)$ for all $j < i$, where $p + t = n$.
    
            Consider the sequence $D = \{a_1^i, a_2^i, \ldots, a_p^i, b_1^i, b_2^i, \ldots, b_t^i\}_{i \in \N}$. It follows from Lemma~\ref{lemma_tuple} that there is a subsequence $\Tilde{D} = \{a_1^{i_1}, a_2^{i_1}, \ldots, a_p^{i_1},\allowbreak b_1^{i_1}, b_2^{i_1}, \ldots, b_t^{i_1}\}_{i \in \N}$ of $D$ such that every column $C_j(\Tilde{D})$, $j = 1, 2, \ldots, n,$ either contains only one element or contains infinitely many elements, and every pair of columns $C_{j_1}(\Tilde{D})$ and $C_{j_2}(\Tilde{D})$ either coincide or have no common elements.
    
            Denote by $J \subset \{1, \ldots, n\}$ the set of indexes of columns of $\Tilde{D}$ consisting of one element. Let  $d = |J|$ and $k = n - d$. Let $P_{\Tilde{D}}^{(k)}$ be the projection of the predicate $P^{(n)}$ onto $J$ by using the corresponding elements from the columns of the sequence $\Tilde{D}$ with indexes in $J$.
    
            Let $I$ be an exact partition of the set $\{1, \ldots, k\}$ such that for every $i,j \in N$ the columns $C_i(\Tilde{D})$ and $C_j(\Tilde{D})$ coincide if and only if the indexes $i,j$ belong to one element of the partition $I$.  Let $P_{\Tilde{D}, I} = \sfrac{P_{\Tilde{D}}^{(k)}}{I}$ be the gluing of $P_{\Tilde{D}}^{(k)}$ by the partition $I$.
    
            Let the arity of the predicate $P_{\Tilde{D}, I}$ be greater than $1$. Then, all the elements from $\Tilde{D}$ define a $P_{\Tilde{D}, I}$~-perfectly non-Noetherian substructure, and, therefore, by Definition~\ref{def_cns}, the structure $\A$ has a $P_{\Tilde{D}, I}$-perfectly non-Noetherian substructure.
    
            Let the arity of the predicate $P_{\Tilde{D}, I}$ be $1$, i.e. $|I| = 1$. Then, it is easy to see that the algebraic structure $\A' = \langle A, Q\rangle$, where $Q = \sfrac{P_{\Tilde{D}}^{(k)}}{\{\{1\}, \{2, \ldots, k\}\}}$, contains a non-Noetherian clique.
   
            Let us prove the opposite direction now. Let an algebraic structure $\A' = \langle A, \sfrac{P}{I} \rangle$ contain a $\sfrac{P}{I}$-perfectly non-Noetherian substructure, where $I = \bigsqcup_{j = 1}^m I_j$ is an exact partition of  $\{1, \ldots, n\}$. Then, by definition, there are sequences $\{(a^i_1, \ldots, a^i_p)\}_{i \in \N}$ and $\{(b^i_1, \ldots, b^i_t)\}_{i \in \N}$, where $p + t = m$, $p \ne 0$, $t \ne 0$, such that $\A \not\models \sfrac{P}{I}(a^i_1, \ldots, a^i_p, b^i_1, \ldots, b^i_t)$ for all $i$, and $\A \models \sfrac{P}{I}(a^i_1, \ldots, a^i_p, b^j_1, \ldots, b^j_t)$ for all $j < i$.  Restore the original predicate $P$ by $\sfrac{P}{I}$ and the exact partition $I$ in the following way. If $i$ and  $j$ are in one element of the partition, then we will identify elements in the $i$-th and $j$-th components. For the obtained sequences, the conditions of Lemma~\ref{lemma:kotov} hold.

            The case when an algebraic structure contains a non-Noetherian clique follows from Remark~\ref{comment:klika}.
        \end{proof}

        Above, we obtained the criterion for predicate structures with one predicate to be equationally Noetherian. The next lemma allows us to generalize the obtained result to languages with an arbitrary finite number of predicate symbols.
        
        \begin{lem}\label{theorem_nP}
            An algebraic structure $\A = \langle A, \La \rangle$ in a predicate language $\La$ with constants from $A$ and a finite number of predicate symbols is not equationally Noetherian if and only if there is a reduction $\La'$ of the language $\La$ such that $\La'$ has only one predicate symbol, and the algebraic structure $\A' = \langle A, \La' \rangle$ is not equationally Noetherian.
        \end{lem}
        
        \begin{proof}
            Since $\A$ is not equationally Noetherian, it follows from Corollary~\ref{cor:kotov} that there are sequences of equations $S(\x)$ and elements $\{a_i\}_{i \in \N}$ satisfying Lemma~\ref{lemma:kotov}. From Proposition~\ref{state:onevariable}, there is an infinite subsystem of equations $S_{P} \subseteq S$ written with only one predicate symbol $P$ of the language $\La$ such  that it is not equivalent to any of its subsystems because of Corollary~\ref{cor:kotov}. Let $\La'$ consist of $P$ and all the constant symbols from $A$. Note that the system $S_{P}$ can be considered over the algebraic structure $\A' = \langle A, \La' \rangle$, where the interpretations of the symbols of the language $\La'$ coincide with the corresponding interpretations of these symbols in the original algebraic structure $\A$. Then $\A'$ is not equationally Noetherian.
        \end{proof}

        Lemmas~\ref{theorem_1P} and~\ref{theorem_nP},  in the form of the following theorem, are a criterion for arbitrary predicate structures to be equationally Noetherian in terms of forbidden substructures.

        \begin{thm}\label{th:main}
            An algebraic structure $\A = \langle A, \La \rangle$ in a predicate language $\La$ with constants from $A$ and a finite number of predicate symbols is not equationally Noetherian if and only if for some predicate symbol $P^{(n)}$ of the language $\La$ there is a projection $P'^{(k)}$ of the predicate $P^{(n)}$ and an exact partition $I$ of $\{1, \ldots, k\}$ such that at least one of the following conditions holds:
            \begin{itemize}
                \item 
                    $|I| > 1$ and the algebraic structure $\A' = \langle A, \sfrac{P'}{I} \rangle$ contains a perfectly non-Noetherian substructure;
                \item 
                    $|I| = 1$ and the algebraic structure $\A' = \langle A, Q \rangle$, where $Q = \sfrac{P'}{\{\{1\}, \{2, \ldots, k\}\}}$, contains a non-Noetherian clique.
            \end{itemize}
       \end{thm}

\section{Equationally Noetherian graphs, hypergraphs, and partial orders}

    In this section, we give examples of the application of Theorem~\ref{th:main} to graphs, partial orders, and hypergraphs.

    Let $\L = \{E^{(2)}\}$ be a language of graph theory. Consider the language $\L_{\Gamma}$ containing constants from the graph $\Gamma$. The statement that nodes $u$ and  $v$ are adjacent can be written in the form of the equation $E(u, v)$ in the language $\L_{\Gamma}$. Note that the equation $E(x, x)$ is always false for simple graphs, and this equation is always true for graphs with loops. Previously, in \cite{BuchinskiyTreier2023}, all equationally Noetherian simple graphs and graphs with loops were described in terms of forbidden subgraphs. The next theorem was the main result of that paper.

    \begin{thm}\label{th:SEMR}
        The following statements are true:
        \begin{itemize}
            \item
                A simple graph is not equationally Noetherian if and only if it is either perfectly non-Noetherian or an overclique.
            \item
                A graph with loops is not equationally Noetherian if and only if it is perfectly non-Noetherian.
        \end{itemize}
    \end{thm}

    Note that the notion of a perfectly non-Noetherian graph coincides with the property of containing a perfectly non-Noetherian substructure, and the notion of being an overclique in the case of simple graphs is equivalent to the property of containing a non-Noetherian clique. Graphs with loops do not contain non-Noetherian cliques because the predicate $E$ is reflexive for such graphs. Therefore, Theorem~\ref{th:SEMR} is a specialization of Theorem~\ref{th:main} for simple graphs and graphs with loops.

    It is easy to see that every predicate structure with one predicate $P^{(n)}$, $n > 2$, can be considered as a hypergraph in which the set of edges is the predicate $P^{(n)}$. Therefore, Theorem~\ref{th:main} can be adapted for hypergraphs.

    Previously, in \cite{NikitinKudyk2018}, a criterion for non-strict partially ordered sets to be equationally Noetherian was proved. The key notions for this criterion are the notions of upper and lower cones. Let us recall them.

    Let $\P = \langle P, \preceq \rangle$ be a partially ordered set, and  $A$ be a subset of it. Let $A^{\uparrow} = \{x \in \P \; | \; \forall a \in A \; a \preceq x\}$ and $A^{\downarrow} = \{x \in \P \; | \; \forall a \in A \; x \preceq a\}$. The pair $(A, A^{\uparrow})$ is called the {\it upper base cone} of  $A$. An upper base cone of $A$ is called {\it finitely generated} if there is a finite subset $B \subseteq A$ such that $B^{\uparrow} = A^{\uparrow}$. Otherwise, if there is no such finite set, we say that it is {\it infinitely generated}. {\it The lower base cone} of $A$ can be defined similarly. 

    The main result of the paper~\cite{NikitinKudyk2018} is the following theorem:

    \begin{thm}\label{theorem_NikitinKudyk}
        A partially ordered set $\P$ is equationally Noetherian if and only if the upper and lower base cones of $A$ are finitely generated for every subset $A$ of $\P$.
    \end{thm}

    Let us show how this theorem is connected to Theorem~\ref{th:main}. Since the predicate $\preceq$ is not symmetric, for partially ordered sets, there exist two perfectly non-Noetherian substructures: for infinite systems of equations of the form $x \preceq b_j$ and infinite systems of equations of the form $b_j \preceq x$. For example, for equations of the form $x \preceq b_j$, a perfect non-Noetherian partially ordered set can be depicted in the following way:

    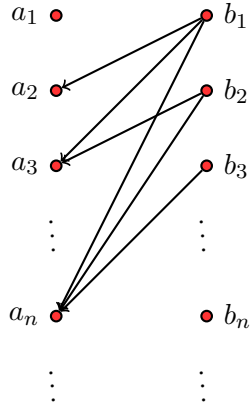
\begin{figure}
        \centering
        \begin{tikzpicture}[thick,scale=1, every node/.style={scale=1}]
            \path (-1,0) node[red node][label=180:$a_1$] (a1) {};
            \path (1,0) node[red node][label=0:$b_1$] (b1) {};
            \path (-1,-1) node[red node][label=180:$a_2$] (a2) {};
            \path (1,-1) node[red node][label=0:$b_2$] (b2) {};
            \path (-1,-2) node[red node][label=180:$a_3$] (a3) {};
            \path (1,-2) node[red node][label=0:$b_3$] (b3) {};
            \path (-0.5,-3) node[][label=180:\rotatebox{130}{$\ddots$}] (ldots1) {};
            \path (1.5,-3) node[][label=180:\rotatebox{130}{$\ddots$}] (rdots1) {};
            \path (-1,-4) node[red node][label=180:$a_n$] (an) {};
            \path (1,-4) node[red node][label=0:$b_n$] (bn) {};
            \path (-0.5,-5) node[][label=180:\rotatebox{130}{$\ddots$}] (ldots2) {};
            \path (1.5,-5) node[][label=180:\rotatebox{130}{$\ddots$}] (rdots2) {};
            \draw [<-](a2) -- (b1);
            \draw [<-](a3) -- (b1);
            \draw [<-](a3) -- (b2);
            \draw [<-](an) -- (b1);
            \draw [<-](an) -- (b2);
            \draw [<-](an) -- (b3);
        \end{tikzpicture}
        \caption{\it An example of a $\prec$-perfectly non-Noetherian substructure; an arc from $b_i$ to $a_j$ means that $a_j \prec b_i$}\label{fig:nonoether}
    \end{figure}

    \begin{rem}
        For all partially ordered sets $\P = \langle P, \preceq \rangle$, there is no non-Noetherian clique because the predicate $\preceq$ is reflexive. Later we will show that for strict partial orders, there is a non-Noetherian clique.
    \end{rem}

    Therefore, the following proposition connects Theorem~\ref{theorem_NikitinKudyk} and the specialization of Theorem~\ref{th:main} for non-strict partially ordered sets:

    \begin{prop}\label{lemma_eq_NikitinKudyk}
        Let $\P = \langle P, \preceq \rangle$ be a partially ordered set. Then the following statements are equivalent:
        \begin{enumerate}
            \item 
                the partial order $P$ is not equationally Noetherian;
            \item
                $\P$ contains a $\preceq$-perfectly non-Noetherian substructure;
            \item
                there is a subset $B$ of $\P$ such that the upper or the lower base cone of $B$ is infinitely generated.
        \end{enumerate}
    \end{prop}

    \begin{proof}
        The equivalence $1 \iff 2$ follows from Theorem~\ref{th:main}. The equivalence $1 \iff 3$ follows from Theorem~\ref{theorem_NikitinKudyk}. Therefore, $2 \iff 3$.
    \end{proof}

    Now, consider strict linear orders. Upper and lower cones can be defined for them similarly. The following proposition is a specialization of Theorem~\ref{th:main} for strict partial orders:

    \begin{prop}    
        Let $\P = \langle P, \prec \rangle$ be a strict partially ordered set. Then the following statements are equivalent:
        \begin{enumerate}
            \item 
                the partial order $P$ is not equationally Noetherian;
            \item
                $\P$ contains a $\prec$-perfectly non-Noetherian substructure or a non-Noetherian clique;
            \item
                there is a subset $B$ of $\P$ such that the upper or lower base cone of $B$ is infinitely generated.
        \end{enumerate}
    \end{prop}

    \begin{proof}
        The equivalence $1 \iff 2$ follows from Theorem~\ref{th:main}. Let us prove that $2 \iff 3$.

        $2 \Rightarrow 3$. Let $\P$ contain a $\prec$-perfectly non-Noetherian substructure, and let $\{a_i\}_{i \in \N}$ and $\{b_i\}_{i \in \N}$ be the sequences defining the perfectly non-Noetherian substructure (see Picture~\ref{fig:nonoether}).  Without loss of generality, we will consider only equations of the form $x \prec b_i$ (otherwise, we would be dealing with equations of the form $b_i \prec x$). Let us show that the lower base cone of the set $B = \{b_i\}_{i \in \N}$ is infinitely generated. Let us assume the converse, i.e., there is a finite subset $C \subsetneq B$ such that $B^{\downarrow} = C^{\downarrow} = \{x \in \P \; | \; \forall c \in C \; x \prec c\}$. Let $C = \{b_{i_1}, \ldots, b_{i_q}\}$ and $m = \max\{i_1, \ldots, i_q\}$. Note that $C^{\downarrow}$ has elements $a_r$ for all $r > m$. They cannot belong to $B^{\downarrow}$ because, for all $a_r, \ r > m$, we have $a_r \not\prec b_r$.  We have a contradiction. For equations of the form $b_i \prec x$, the reasoning is similar with the difference that it is necessary to consider the upper base cone of the set $B$. 

        Let $\P$ contain a non-Noetherian clique, i.e., there is a sequence $\{a_i\}_{i \in \N} \subseteq P$ such that $\P \not\models a_i \prec a_i$ for all $i$ (it is always true because the predicate $\prec$ is reflexive) and $\P \models a_i \prec a_j$ for all $j < i$. As in the case of a $\prec$-perfectly non-Noetherian substructure, show that the lower base cone of $B = \{a_i\}_{i \in \N}$ is infinitely generated. Assume the converse, i. e. there is a finite subset $C \subsetneq B$ such that $B^{\downarrow} = C^{\downarrow} = \{x \in \P \; | \; \forall c \in C \; x \prec c\}$. Let $C = \{a_{i_1}, \ldots, a_{i_q}\}$ and $m = \max\{i_1, \ldots, i_q\}$. Note that $C^{\downarrow}$ has elements $a_r$ for all $r > m$, which do not belong to $B^{\downarrow}$. Contradiction. For equations of the form $a_i \prec x$, the reasoning is similar with the difference that it is necessary to consider the upper base cone of $B$.

        $3 \Rightarrow 2$. Let $B$ be a subset of $\P$ such that the lower base cone of $B$ is infinitely generated. Let $b_0$ be an element of $B$. Since the lower base cone of $B$ is infinitely generated, there is an element $b_1 \in B$ such that $\{b_0, b_1\}^{\downarrow} \subsetneq \{b_0\}^{\downarrow}$. Fix $a_1 \in \{b_0\}^{\downarrow} \setminus \{b_0, b_1\}^{\downarrow}$. Because the lower base cone of $B$ is infinitely generated, there is an element $b_2 \in B$ such that $\{b_0, b_1, b_2\}^{\downarrow} \subsetneq \{b_0, b_1\}^{\downarrow}$. Fix $a_2 \in \{b_0, b_1\}^{\downarrow} \setminus \{b_0, b_1, b_2\}^{\downarrow}$. Continuing this procedure, we obtain sequences of elements $\{a_i\}_{i \in \N}$ and $\{b_i\}_{i \in \N}$ such that $\P \not\models a_i \prec b_i$ for all $i$, and $\P \models a_i \prec b_j$ for $j < i$.
    
        Note that for all different $i$ and $j$, $a_i \ne a_j$ and $b_i \ne b_j$. Then, by Lemma~\ref{lemma_tuple}, for the sequence $\{(a_i, b_i)\}_{i \in \N}$, there is a subsequence $I \subseteq \N$ such that the sequences $\{a_i\}_{i \in I}$ and $\{b_i\}_{i \in I}$ either contain pairwise different elements or contain elements such that $a_i = b_i$ for all $i \in I$. In the first case, $\P$ contains a perfectly non-Noetherian substructure. In the second case, the sequences $\{a_i\}_{i \in I}$ and $\{b_i\}_{i \in I}$ form a non-Noetherian clique.

        If $B$ is a subset such that its upper base cone is infinitely generated, the reasoning is similar with the difference that it is necessary to use the predicate~$\succ$.
    \end{proof}

    \begin{exa}\label{last_example}
        Consider the natural order of integers $\mathcal{Z} = \langle \mathbb{Z}, \prec \rangle$. Note that the conditions from Definition~\ref{def:clique} hold for the sequence $\{i\}_{i \in \N}$, i.e., the order $\mathcal{Z}$ contains a non-Noetherian clique. The sequences of elements $\{-2i\}_{i \in \N}$ and equations $\{x \prec -2i - 1\}_{i \in \N}$, by Definition~\ref{def_cns}, define a perfectly non-Noetherian substructure of $\mathcal{Z}$. 
    
        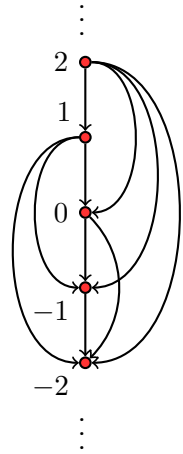
\begin{figure}
            \centering
            \begin{tikzpicture}[thick,scale=1, every node/.style={scale=1}]
                \path (-0.5,-3.5) node[][label=180:\rotatebox{125}{$\ddots$}] (ldots1) {};
                \path (-1,-6) node[red node][label=180:$0$] (a) {};
                \path (-1,-5) node[red node][label=120:$1$] (au) {};
                \path (-1,-4) node[red node][label=180:$2$] (auu) {};
                \path (-1,-7) node[red node][label=210:$-1$] (al) {};
                \path (-1,-8) node[red node][label=210:$-2$] (all) {};
                \path (-0.5,-9) node[][label=180:\rotatebox{125}{$\ddots$}] (ldots2) {};
                \draw [->](auu) -- (au);
                \draw [->](auu) edge[bend left=90] node {} (a);
                \draw [->](auu) edge[bend left=90] node {} (al);
                \draw [->](auu) edge[bend left=90] node {} (all);
                \draw [->](au) -- (a);
                \draw [->](au) edge[bend right=90] node {} (al);
                \draw [->](au) edge[bend right=90] node {} (all);
                \draw [->](a) -- (al);
                \draw [->](a) edge[bend left=45] node {} (all);
                \draw [->](al) -- (all);
            \end{tikzpicture}
            \caption{The partial order $\mathcal{Z} = \langle \mathbb{Z}, \prec \rangle$ from Example~\ref{last_example}}
        \end{figure}
    \end{exa}

\section*{Acknowledgments}
This research was supported in accordance with the state task of the IM SB RAS, project FWNF-2022-0003.

\bibliographystyle{plain}
\bibliography{main}

\begin{thebibliography}{10}

\bibitem{BaumslagMyasnikovRemeslennikov1999}
G.~Baumslag, A.~Myasnikov, and V.~Remeslennikov.
\newblock Algebraic geometry over groups {I}: Algebraic sets and ideal theory.
\newblock {\em J. Algebra}, 219:16--79, 1999.

\bibitem{BaumslagMyasnikovRomankov1997}
G.~Baumslag, A.~Myasnikov, and V.~Roman'kov.
\newblock Two theorems about equationally {N}oetherian groups.
\newblock {\em J. Algebra}, 194:654--664, 1997.

\bibitem{BuchinskiyTreier2023}
I.~M. Buchinskiy and Treier~A. V.
\newblock On graphs that are not equationally {N}oetherian.
\newblock {\em SEMR}, 20(2):580--587, 2023.

\bibitem{DaniyarovaMyasnikovRemeslennikov2016}
E.~Daniyarova, A.~Miasnikov, and V.~Remeslennikov.
\newblock {\em Algebraic geometry over algebraic structures}.
\newblock Publishing House of SB RAS, Novosibirsk, 2016.
\newblock (in Russian).

\bibitem{DaniyarovaMyasnikovRemeslennikov2008}
E.~Daniyarova, A.~Myasnikov, and V.~Remeslennikov.
\newblock Unification theorems in algebraic geometry.
\newblock {\em Algebra and Discrete Mathematics}, 1:80--112, 2008.

\bibitem{DaniyarovaMyasnikovRemeslennikov2011}
E.~Daniyarova, A.~Myasnikov, and V.~Remeslennikov.
\newblock Algebraic geometry over algebraic structures {III}: Equationally
  {N}oetherian property and compactness.
\newblock {\em Southeast Asian Bulletin Math.}, 35(1):35--68, 2011.

\bibitem{DaniyarovaMyasnikovRemeslennikov2012_1}
E.~Yu. Daniyarova, A.~G Miasnikov, and V.~N. Remeslennikov.
\newblock Algebraic geometry over algebraic structures {II}: Foundations.
\newblock {\em Fund. Prikl. Math}, 17(1):65--106, 2012.

\bibitem{DaniyarovaMyasnikovRemeslennikov2012_2}
E.~Yu. Daniyarova, A.~G. Myasnikov, and V.~N. Remeslennikov.
\newblock Algebraic geometry over algebraic structures. {V}. the case of
  arbitrary signature.
\newblock {\em Algebra Logic}, 51(1):28--40, 2012.

\bibitem{Date2003}
C.~J. Date.
\newblock {\em An Introduction to Database Systems}.
\newblock Pearson, 8 edition, 2003.

\bibitem{Dvorzhetskiy2013}
Yu.~S. Dvorzhetskiy.
\newblock Algebraic geometry over lattices with distinguished ideals.
\newblock {\em Herald Omsk Univer.}, 70(4):30--35, 2013.

\bibitem{DvorzhetskiyKotov2008}
Yu.~S. Dvorzhetskiy and Kotov~M. V.
\newblock Minimax algebraic structures.
\newblock {\em Vestnik of Omsk University}, Combinatorial Methods of Alg. and
  Comput. Complexity:130--136, 2008.

\bibitem{GuptaRomanovskii2007}
Ch.~K. Gupta and N.~S. Romanovskii.
\newblock The property of being equationally {N}oetherian for some soluble
  groups.
\newblock {\em Algebra and Logic}, 46(1):46--59, 2007.

\bibitem{IljevRemeslennikov2017}
A.~V. Iljev and V.~N. Remeslennikov.
\newblock Study of the compatibility of systems of equations over graphs and
  finding their general solutions.
\newblock {\em Herald of Omsk University}, 4(86):26--32, 2017.

\bibitem{Kotov2011}
M.~Kotov.
\newblock Equationally {N}oetherian property and close properties.
\newblock {\em Southeast Asian Bulletin of Mathematics}, 35(3):419--429, 2011.

\bibitem{Kotov2013_1}
M.~V. Kotov.
\newblock Several remarks on equationally {N}oetherian property.
\newblock {\em Herald of Omsk University}, 2:24--28, 2013.

\bibitem{Kotov2013_2}
M.~V. Kotov.
\newblock Topologizability of countable equationally {N}oetherian algebras.
\newblock {\em Algebra and Logic}, 52(2):105--115, 2013.

\bibitem{ModabberiShahryari2016}
P.~Modabberi and M.~Shahryari.
\newblock Comapactness conditions in universal algebraic geometry.
\newblock {\em Algebra i Logika}, 55(2):219--256, 2016.

\bibitem{MyasnikovRemeslennikov2000}
A.~Myasnikov and V.~Remeslennikov.
\newblock Algebraic geometry over groups {II}: logical foundations.
\newblock {\em J. Algebra}, 234:225--276, 2000.

\bibitem{NikitinKudyk2018}
A.~Yu. Nikitin and I.~D. Kudyk.
\newblock Criterion of equationally {N}oetherian property for posets.
\newblock In {\em Journal of Physics: Conf. Ser.}, volume 1050:1, page 012058.
  IOP Publishing, 2018.

\bibitem{NikitinShevlyakov2020}
A.~Yu. Nikitin and A.~N. Shevlyakov.
\newblock On radicals of a system of equations over linear strict posets.
\newblock In {\em J. Phys.: Conf. Ser.}, volume 1441:1, page 012156, 2020.

\bibitem{NikitinShevlyakov2021}
A.~Yu. Nikitin and A.~N. Shevlyakov.
\newblock On radicals over strict partial order sets.
\newblock In {\em J. Phys.: Conf. Ser.}, volume 1791:1, page 012080, 2021.

\bibitem{Plotkin1997}
B.~Plotkin.
\newblock Varieties of algebras and algebraic varieties. categories of
  algebraic varieties.
\newblock {\em Siberian Advances in Math.}, 7(2):64--97, 1997.

\bibitem{Plotkin2003}
B.~Plotkin.
\newblock Algebras with the same (algebraic) geometry.
\newblock {\em Proc. Steklov Inst. Math.}, 242:165--196, 2003.

\bibitem{ShahryariShevlyakov2017}
M.~Shahryari and A.~Shevlyakov.
\newblock Direct products, varieties, and compactness conditions.
\newblock {\em Groups Complexity Cryptology}, 9(2):159--166, 2017.

\bibitem{Shevlyakov2011}
A.~N. Shevlyakov.
\newblock Commutative idempotent semigroups at the service of the universal
  algebraic geometry.
\newblock {\em Southeast Asian Bulletin Math.}, 35(1):111--136, 2011.

\bibitem{Shevlyakov2016}
A.~N. Shevlyakov.
\newblock Universal algebraic geometry with relation $\neq$.
\newblock {\em Algebra i Logika}, 55(4):498--511, 2016.

\bibitem{Shevlyakov2018}
A.~N. Shevlyakov.
\newblock Algebraic geometry over groups in predicate language.
\newblock {\em Herald of Omsk University}, 23(4):60--63, 2018.

\bibitem{Shevlyakov2021}
A.~N. Shevlyakov.
\newblock Equations over direct powers of algebraic structures in relational
  languages.
\newblock {\em Prikl. Diskr. Mat.}, 53:5--11, 2021.

\end{thebibliography}

\end{document}